\documentclass[11pt,leqno]{amsart}
\usepackage{amsthm,amsfonts,amssymb,amsmath,amscd,latexsym}
\usepackage{graphicx,exscale,cite,epsfig}

\usepackage{hyperref,xcolor}
\usepackage{fullpage}

\renewcommand{\geq}{\geqslant}
\renewcommand{\leq}{\leqslant}

\def\eps{\varepsilon}

\newcommand\ba{\begin{equation}\begin{aligned}}
\newcommand\ea{\end{aligned}\end{equation}}
\newcommand{\be}{\begin{equation}}
\newcommand{\ee}{\end{equation}}

\newtheorem{theorem}{Theorem}[section]

\newtheorem{corollary}[theorem]{Corollary}
\newtheorem{lemma}[theorem]{Lemma}
\newtheorem{definition}[theorem]{Definition}
\newtheorem*{definition*}{Definition}

\theoremstyle{remark}
\newtheorem{remark}{Remark}
\newtheorem*{remark*}{Remark}

\numberwithin{equation}{section}

\begin{document}

\title[An alternative proof of modulational instability of Stokes waves in deep water]{An alternative proof of modulational instability of Stokes waves in deep water}

\author{Zhao Yang}
\email{zhaouiuc@illinois.edu}

\date{\today}

\keywords{Stokes wave; stability; spectrum; Benjamin--Feir; periodic Evans function}

\begin{abstract}
We generalize the periodic Evans function approach recently used to study the spectral stability of Stokes wave and gravity-capillary (including Wilton ripples) in water of finite depth to study spectral stability of Stokes waves in water of infinite depth. We prove waves of sufficiently small amplitude are always low-frequency unstable regardless of the wave number and gravity, giving an alternative proof for the Benjamin-Feir modulational instability in the infinite depth case. Here, the first proof for the infinite depth case is recently obtained by Nguyen and Strauss. We also study the spectral stability at non-zero resonant frequencies and find no additional instability.
\end{abstract}

\maketitle 

\tableofcontents

\newpage

\section{Introduction}\label{sec:intro}

Consider water wave equations in two dimensions with infinite depth
\ba\label{eqn:ww;h}
&\;\phi_{xx}+\phi_{yy}=0&& \text{for $-\infty<y<\eta(x,t)$,}\\
&\;\phi_y=0&&\text{as $y\rightarrow -\infty$,}\\
&\left.\begin{aligned}&\eta_t-c\eta_x+\eta_x\phi_x=\phi_y\\
&\phi_t-c\phi_x+\frac12(\phi_x^2+\phi_y^2)+g\eta=0\end{aligned}\right\}&&\text{at $y= \eta(x,t)$.}
\ea
where $\phi(t,x,y)$ is a velocity potential, $y=\eta(x,t)$ denotes the free surface, $c$ is the speed of a moving reference frame, $g$ is the gravity constant. System \eqref{eqn:ww;h} is commonly used to model the wave motion at the free surface of
an incompressible inviscid fluid of infinite depth in two dimensions, lying below a body of air, acted on by gravity, when the effects of surface tension are negligible. In the case of fluid of finite depth, the two $-\infty$ appearing in the first two equations of \eqref{eqn:ww;h} are replaced by $-h$ which is the depth of the fluid at rest. It has been long known that, for either the finite or infinite depth case, system \eqref{eqn:ww;h} admits periodic traveling wave trains known as Stokes waves \cite{Stokes1847}.

For the case of finite depth, Benjamin and Feir \cite{Benjamin;BF,BF;BF,Whitham;BF} discovered that a $2\pi/\kappa$-periodic Stokes wave of mean depth $h$ is subject to the  modulational instability when $1.3627\cdots<\kappa h$. Nearly three decades later, a rigorous proof of the instability appeared in the work of Bridges and Mielke \cite{BM;BF}. Recently, numerical investigations by Deconinck and Oliveras \cite{DO} suggested that Stokes waves may be subject to additional instabilities other than the modulational instability. To prove the additional instabilities and obtain thresholds for its onset by rigorous spectral analysis, V. Hur and Z. Yang \cite{hur2021unstable} developed a new periodic Evans function approach to rigorously study the spectral stability of Stokes waves. The approach allowed them to 
establish an alternative proof for the modulational instability. Moreover, when it was applied to analyze spectra away from the origin, they proved a $2\pi/\kappa$-periodic Stokes wave is spectral unstable at the resonant frequency of order $2$, provided that $0.86430\cdots<\kappa<1.00804\cdots$, justifying the additional instability suggested by numeric. The new approach is very robust and was generalized to the situation in which surface tension or vorticity effect is considered.

For the case of infinite depth, the Benjamin and Feir instability for all waves was recently proven by Nguyen and Strauss \cite{nguyen2021proof}. They wrote the water wave equations \eqref{eqn:ww;h} in the Zakharov-Craig-Sulem formulation which involves the non-local Dirichlet-Neumann operator. In the linearized problem, they used Alinhac’s good
unknown and made the free surface flat by using the conformal mapping between the fluid domain and
the lower half-plane, which also converts the implicit non-local operator to an explicit Fourier multiplier. The task of proving instability then reduces to finding an eigenvalue $\lambda(\mu,\eps)$ of $\mathcal{L}_{\mu,\eps}$ with positive real part, which they completed by first studying $\mathcal{L}_{0,\eps}$ and then $\mathcal{L}_{\mu,\eps}$ by asymptotic analysis.

Our goal in this piece of work is to generalize the periodic Evans function approach \cite{hur2021unstable,hur2021unstable_cap} to the case of infinite depth. In order to formulate the linearized equations as first order ODEs with respect to the $x$ variable (see \eqref{eqn:linearize}), we again introduce
\be\label{def:u}
u=\phi_x.
\ee 
To flatten the free surface, instead of making change of coordinate $y\mapsto \frac{y-\eta(x,t)}{-h-\eta(x,t)}$ for the finite depth case \cite{hur2021unstable,hur2021unstable_cap}, we perform change of coordinate
\begin{equation}\label{def:y}
 y\mapsto y-\eta(x,t),
\end{equation}
transforming the fluid region $\{(x,y)\in\mathbb{R}^2: -\infty<y<\eta(x,t)\}$ into $\mathbb{R}\times (-\infty,0)$. Substituting \eqref{def:u} and \eqref{def:y} into \eqref{eqn:ww;h}, we use the chain rule and make a straightforward calculation to arrive at

\ba\label{eqn:ww}
&\phi_x-\eta_x\phi_y-u=0&&\text{for $y<0$,}\\
&u_x-\eta_x u_y+\phi_{yy}=0&&\text{for $y<0$,}\\
&\phi_y=0&&\text{as $y\rightarrow -\infty$},\\
&\eta_t+(u-c)\eta_x-\phi_y=0&&\text{at $y=0$},\\
&\phi_t-cu+(u-c)\eta_x\phi_y+\frac12u^2-\frac{1}{2}\phi_y^2+g\eta=0\quad&&\text{at $y=0$}.
\ea 
In section \ref{sec:Stokes}, we set the speed $c$ of the moving frame to be the speed of a Stokes wave and compute asymptotic expansions for the wave. In section \ref{sec:spec}, we linearize \eqref{eqn:ww} about a Stoke wave $(\phi(\eps)$, $u(\eps)$, $\eta(\eps))$ and write corresponding spectral problem in concise form
\be 
\label{eqn:LB_intro}
\mathbf{u}_x=\mathbf{L}(\lambda)\mathbf{u}+\mathbf{B}(x;\lambda,\eps)\mathbf{u},
\ee 
where $\mathbf{u}$ is the perturbation of $(\phi,u,\eta)^T$. When the fluid is at rest, the velocity potential is any constant and $u$ and $\eta$ are both zero. That is $\mathbf{u}=[1,0,0]^T$ solves $\mathbf{u}_x=\mathbf{L}(0)\mathbf{u}$. It is then necessary to include $\mathbf{u}=[1,0,0]^T$ in the domain of $\mathbf{L}(0)$. We therefore define a function space $H_c^n(-\infty,0)$, $n\ge 0$ by 
\be 
H^n_c(-\infty,0):=\{f+z:f\in H^n(-\infty,0), \;z \in \mathbb{C}\}
\ee
where $H^n(-\infty,0)$ is the standard Sobolev space.
For $f\in H^n_c(-\infty,0)$, we see there is a unique complex number denoted as $f_{\infty}$ so that $f-f_{\infty}\in H^n(-\infty,0)$. We equip the space $H^n_c(-\infty,0)$ with inner product 
$$
\langle f,g\rangle=\langle f-f_{\infty},g-g_{\infty}\rangle_{H^n(-\infty,0)}+f_{\infty}\overline{g_{\infty}},
$$
so that the map $T:H^n(-\infty,0)\times \mathbb{C}\rightarrow H^n_c(-\infty,0)$ defined by $T(f,c)=f+c$ is an isometry. Also, because of the isometry,  $\big(H^n_c(-\infty,0),\langle \cdot,\cdot\rangle\big)$ is a Hilbert space. The periodic Evans function approach then applies to study \eqref{eqn:LB_intro} when $\mathbf{L}(\lambda)$ and $\mathbf{B}(x;\lambda,\eps)$ are consider as unbounded operators on 
\begin{equation}\label{def:Y}
Y=H^1_c(-\infty,0)\times L_c^2(-\infty,0)\times \mathbb{C}, \quad L_c^2(-\infty,0):=H^0_c(-\infty,0).
\end{equation}
For simplicity, we briefly recap the reduction of spectral problem and definition of periodic Evans function, but refer readers to section 4 of \cite{hur2021unstable} or section 3.4 of \cite{hur2021unstable_cap} for more details. Our focus here is to quickly carried out the key computations necessary for asymptotic expansion of the periodic Evans function. By the expansion of the periodic Evans function, we obtain our main results
\begin{theorem}[Spectral instability near $0\in\mathbb{C}$ at $\gamma\eps$ order]\label{thm:BF}
There are two algebraic curves $\lambda_{1,2}(k,\eps)$ of the essential spectra of a periodic Stokes wave of sufficiently small amplitude in deep water bifurcate off the imaginary axis near the origin. Indeed, the two algebraic curves $\lambda_{1,2}(k,\eps)$ expand as
\be 
\lambda_{1,2}(\pm\kappa+\gamma,\eps)=\frac{ic_0}{2}\gamma\pm\frac{\kappa}{2\sqrt{2}}\gamma\eps+o(|\gamma|+|\gamma\eps|),
\ee 
where $\gamma$ is a perturbation of the Floquet exponent $k$ and $\eps$ is the parameter for wave amplitude, $c_0$ is the speed of the wave, $\kappa=g/c_0^2$ is the wave number of the periodic wave, yielding spectral instability of a periodic Stoke wave of sufficiently small amplitude in deep water.
\end{theorem}
\begin{theorem}[Spectral instability away from $0\in\mathbb{C}$ at $\eps^2$ order ]\label{thm:unstable2}
At the resonant frequency $\lambda=i\frac{N^2-1}{4}\kappa c_0$ where $k_2(\sigma)-k_4(\sigma)=N\kappa$, the corresponding Stoke wave of sufficiently small amplitude is spectrally unstable near the resonant frequency $\lambda=i\sigma$ where $k_2(\sigma)-k_4(\sigma)=N\kappa$, at the order of $\eps^2$ as $\eps\to 0$, provided that 
\be \label{def:ind2}
{\rm ind}_2(\beta,\kappa,\sigma,k_2(\sigma),k_4(\sigma)):=\frac{(a_{11}^{(0,2)}a_{22}^{(1,0)}-a_{11}^{(1,0)}a_{22}^{(0,2)})^2}{(a_{11}^{(1,0)}a_{22}^{(1,0)})^2}>0.
\ee
Indeed, it can be shown ${\rm ind}_2$ is always negative, implying there is no instability at these resonant frequencies at $\eps^2$ order.
\end{theorem}
\begin{remark}
In water of finite depth, we found in \cite{hur2021unstable} that a $2\pi/\kappa$-periodic Stokes wave with mean depth $h$ is spectrally unstable at the resonant frequency of order $2$, provided that $0.86430\cdots<\kappa h<1.00804\cdots.$
\end{remark}
\section{Stokes waves of sufficiently small amplitude in deep water}\label{sec:Stokes}
Let $(\phi(\eps),\eta(\eps),c(\eps))$ and hence $u(\eps)=\phi_x(\eps)-\eta_x(\eps)\phi_y(\eps)$ be a temporally stationary and spatially periodic solution of \eqref{eqn:ww}
and suppose that they expand as
\begin{equation}\label{eqn:stokes exp}
\begin{aligned}
\phi(x,y;\eps)=&\phi_1(x,y)\eps+\phi_2(x,y)\eps^2 +\phi_3(x,y)\eps^3+O(\eps^4), \\
\eta(x;\eps)=&\eta_1(x)\eps+\eta_2(x)\eps^2+\eta_3(x)\eps^3+O(\eps^4),\\
c(\eps)=&c_0+c_1\eps+c_2\eps^2+c_3\eps^3+O(\eps^4),
\end{aligned}
\end{equation}
as $\eps\to0$, where $\phi_1,\phi_2,\phi_3,\dots$, $\eta_1,\eta_2,\eta_3,\dots$, $c_0,c_1,c_2,c_3,\dots$ are to be determined. We assume that $\phi_1,\phi_2,\phi_3,\dots$ and  $\eta_1,\eta_2,\eta_3,\dots$ are $T$ periodic functions of $x$, where 
\[
\text{$T=2\pi/\kappa$ and $\kappa>0$ is the wave number,} 
\]
and $c_0,c_1,c_2,c_3,\dots$ are constants. We may also assume that $\phi_1,\phi_2,\phi_3,\dots$ are odd functions of $x$, and $\eta_1,\eta_2,\eta_3,\dots$ are even functions. Substituting \eqref{eqn:stokes exp} into \eqref{eqn:ww}, at the order of $\eps$, we gather
\ba \label{eqn:stokes1}
&{\phi_1}_{xx}+{\phi_1}_{yy}=0&&\text{for $y<0$},\\
&{\phi_1}_y=0&&\text{as $y\rightarrow -\infty$},\\
&c_0{\eta_1}_x+{\phi_1}_y=0 &&\text{at $y=0$},\\
&c_0{\phi_1}_x-g\eta_1=0\quad&&\text{at $y=0$}.
\ea
We solve \eqref{eqn:stokes1} by separation of variables to obtain
\be\label{def:stokes1} 
\phi_1(x,y)=\sin(\kappa x)e^{\kappa y},  \quad \eta_1(x)=\frac{1}{c_0} \cos(\kappa x),\quad g=\kappa c_0^2. 
\ee
We proceed likewise, substituting \eqref{eqn:stokes exp} into \eqref{eqn:ww}, and solving at higher orders of $\eps$, to successively obtain $\phi_2,\phi_3,\dots$, $\eta_2,\eta_3,\cdots$, and $c_1,c_2,\cdots$. The results are 
\be 
\label{profile_higher}
\phi_2=\frac{1}{2c_0}\kappa e^{\kappa y}\sin(2\kappa x), \quad \eta_2=\frac{1}{2c_0^2}\kappa\cos(2\kappa x), \quad c_1=0,\quad c_2=\frac{\kappa^2}{2c_0}.
\ee 
It turns out \eqref{def:stokes1}  together with \eqref{profile_higher} are enough for our analysis.
\section{The spectral stability problem}\label{sec:spec}

\subsection{The linearized problem}\label{sec:linearize}

Linearizing \eqref{eqn:ww} about $\phi(\eps)$, $u(\eps)$, $\eta(\eps)$ and evaluating the result at $c=c(\eps)$, we arrive at
\ba \label{eqn:linearize}
&\phi_x-\eta_x(\eps)\phi_y-\phi_y(\eps)\eta_x-u=0&&\text{for $y<0$,}\\
&u_x-\eta_x(\eps)u_y-u_y(\eps)\eta_x+\phi_{yy}=0&&\text{for $y<0$,}\\
&\eta_t+(u(\eps)-c)\eta_x+\eta_x(\eps)u-\phi_y=0&&\text{at $y=0$},\\
&\phi_t-cu+((u-c)\phi_y)(\eps)\eta_x+(\eta_x\phi_y)(\eps)u+u(\eps)u+g\eta=0&&\text{at $y=0$},\\
&\phi_y=0&&\text{as $y\rightarrow -\infty$},\\
\ea 
Seeking a solution of \eqref{eqn:linearize} of the form 
\[
\begin{pmatrix}\phi(x,y,t) \\ u(x,y,t) \\ \eta(x,t) \end{pmatrix}
=e^{\lambda t}\begin{pmatrix}\phi(x,y) \\ u(x,y) \\ \eta(x) \end{pmatrix},
\quad\lambda\in\mathbb{C},
\] 
the corresponding spectral problem reads
\begin{subequations}\label{eqn:spec}
\begin{align}
&\phi_x-\eta_x(\eps)\phi_y-\phi_y(\eps)\eta_x-u=0&&\text{for $y<0$,} \label{eqn:spec;phi}\\
&u_x-\eta_x(\eps)u_y-u_y(\eps)\eta_x+\phi_{yy}=0&&\text{for $y<0$,}\label{eqn:spec;u}\\
&\lambda\eta+(u(\eps)-c)\eta_x+\eta_x(\eps)u-\phi_y=0&&\text{at $y=0$},\label{eqn:spec;K}
\intertext{and}
\label{eqn:spec;D}&\lambda\phi+(u-c)(\eps)u+(g-\lambda\phi_y(\eps))\eta+\phi_y(\eps)\phi_y=0&&\text{at $y=0$},\\
&\phi_y=0&&\text{as $y\rightarrow -\infty$},\label{eqn:spec;bdry}
\end{align}
\end{subequations}
where \eqref{eqn:spec;D} follows from \eqref{eqn:spec;K} and
\[
\lambda\phi-cu+((u-c)\phi_y)(\eps)\eta_x+(\eta_x\phi_y)(\eps)u+u(\eps)u+g\eta=0
\quad\text{at $y=0$},
\] 
by the fourth equation of \eqref{eqn:linearize}. Notice that \eqref{eqn:spec;D} is not autonomous. Thus we introduce
\begin{equation}\label{def:tildeu}
\begin{aligned}
\tilde{u}=&\big((c(\eps)-u(\cdot,0;\eps))u-\lambda\phi-\phi_y(\cdot,0;\eps)\phi_y\big)(g-\lambda\phi_y(\cdot,0;\eps))^{-1}
\end{aligned}
\end{equation}
so that \eqref{eqn:spec;D} becomes 
\begin{equation}\label{eqn:bdry;zeta}
\eta-\tilde{u}=0\quad\text{at $y=0$}.
\end{equation}
Clearly, for $\lambda\in\mathbb{C}$, \eqref{def:tildeu} is well defined for $\eps\in\mathbb{R}$ and $|\eps|\ll1$. Conversely,
\begin{equation} \label{def:u(tildeu)}
\begin{aligned}
u=\big((g-\lambda\phi_y(\cdot,0;\eps))\tilde{u}+\lambda\phi+\phi_y(\cdot,0;\eps)\phi_y\big)\big(c(\eps)-u(\cdot,0;\eps)\big)^{-1}
\end{aligned}
\end{equation}
is well defined, provided that $\eps\in\mathbb{R}$ and $|\eps|\ll1$.
We then write \eqref{eqn:spec;phi}-\eqref{eqn:spec;K} also in the unknowns $\left[\begin{array}{rrr}\phi&\tilde{u} &\eta\end{array}\right]$. That is, we replace the term $u(x,y)$ appears in \eqref{eqn:spec;phi} by the right hand side of \eqref{def:u(tildeu)}, the term $u(\cdot,1)$ appears in \eqref{eqn:spec;K} by the right hand side of \eqref{def:u(tildeu)} evaluating at $y=1$, the term $u_y(x,y)$ appears in \eqref{eqn:spec;u} by the right hand side of
\begin{equation} \label{def:u(tildeu)_y}
\begin{aligned}
u_y=\big((g-\lambda\phi_y(\cdot,0;\eps))\tilde{u}_y+\lambda\phi_y+\phi_y(\cdot,0;\eps)\phi_{yy}\big)\big(c(\eps)-u(\cdot,0;\eps)\big)^{-1}
\end{aligned}
\end{equation}
and finally the term $u_x(x,y)$ appears in \eqref{eqn:spec;u} by the right hand side of
\begin{equation} \label{def:u_x}
\begin{aligned}
u_x(x,y)=&f_1'(x)\tilde{u}(x,y)+f_1(x)\tilde{u}_x(x,y)+\lambda f_2'(x)\phi+\lambda f_2(x)\phi_x\\
&+f_3'(x)\phi_y(x,y)+f_3(x)\phi_{xy}(x,y)
\end{aligned}
\end{equation}
where 
$$
\begin{aligned}
f_1(x):=&(g-\lambda\phi_y(\cdot,0;\eps))f_2(x),\\
f_2(x):=&\big(c(\eps)-u(\cdot,0;\eps)\big)^{-1},\\
f_3(x):=&\phi_y(\cdot,0;\eps)f_2(x),
\end{aligned}
$$
and $\phi_x(x,y)$ is replaced by the newly obtained \eqref{eqn:spec;u} in $\left[\begin{array}{rrr}\phi&\tilde{u} &\eta\end{array}\right]$ unknowns and $\phi_{xy}(x,y)$ is replaced by differentiating the newly obtained \eqref{eqn:spec;u} against $y$.
\subsection{Spectral stability and instability}\label{sec:L}

Let $\mathbf{u}=\begin{pmatrix} \phi \\ \tilde{u} \\ \eta \end{pmatrix}$, for convenience we drop the $\tilde{}$ on $\tilde{u}$ from now on,
and we write \eqref{eqn:spec} in concise form \eqref{eqn:LB_intro} where $\mathbf{L}(\lambda): {\rm dom}(\mathbf{L}) \subset Y \to Y$ is given by
\ba \label{def:L}
&\mathbf{L}(\lambda)\mathbf{u}=\begin{pmatrix} (\lambda \phi+g u)/c_0\\ -(c_0^2\phi_{yy}+\lambda^2\phi+g\lambda u)/(gc_0)\\
(\lambda\eta-\phi_y(0))/c_0\end{pmatrix},\\
&{\rm dom}(\mathbf{L})
=\{\mathbf{u}\in H^2_c(-\infty,0)\times H^1_c(-\infty,0)\times \mathbb{C}:\eta-u(0)=0, \phi_y(-\infty)=0\}
\ea 
and $\mathbf{B}(x;\lambda,\eps): \mathbb{R}\times {\rm dom}(\mathbf{L}) \subset \mathbb{R}\times Y \to Y$ is the higher $O(\eps)$ part of \eqref{eqn:spec;phi}-\eqref{eqn:spec;K}.
Let 
\[
\mathcal{L}(\lambda,\eps):{\rm dom}(\mathcal{L})\subset X \to X,
\] 
where 
\begin{equation}\label{def:operator}
\mathcal{L}(\lambda,\eps)\mathbf{u}=\mathbf{u}_x-(\mathbf{L}(\lambda)+\mathbf{B}(x;\lambda,\eps))\mathbf{u},
\end{equation}
\begin{equation}\label{def:X}
X=L^2(\mathbb{R};Y)\quad\text{and}\quad 
{\rm dom}(\mathcal{L})=H^1(\mathbb{R};Y)\bigcap L^2(\mathbb{R};{\rm dom}(\mathbf{L}))
\end{equation}
is dense in $X$, so that \eqref{eqn:LB_intro} becomes 
\begin{equation}\label{eqn:L}
\mathcal{L}(\lambda,\eps)\mathbf{u}=0.
\end{equation}
We regard $\mathcal{L}(\eps)$ as $\mathcal{L}(\lambda,\eps)$, parametrized by $\lambda\in\mathbb{C}$. We then follow section 4 of \cite{hur2021unstable} or section 3.4 of \cite{hur2021unstable_cap} to define spectrum of $\mathcal{L}(\eps)$, make the center manifold reduction, and define a periodic Evans function. We collect the key computation results below.

\subsection{The spectrum of \texorpdfstring{$\mathcal{L}(0)$}{Lg}. The reduced space}\label{sec:eps=0}
When $\eps=0$ in \eqref{eqn:LB_intro},
solving
\[
ik\mathbf{u}=\mathbf{L}(\lambda)\mathbf{u},
\quad\text{for some $k\in\mathbb{R}$}
\] 
reveals that 
\begin{equation}\label{eqn:sigma}
\lambda=i\sigma,\quad\text{where}\quad
\sigma\in\mathbb{R}\quad\text{and}\quad(\sigma-c_0k)^2=g|k|.
\end{equation}
Thus ${\rm spec}(\mathcal{L}(0))=i\mathbb{R}$, implying that the Stokes wave of zero amplitude is spectrally stable. 
Let
\begin{equation}\label{def:sigma}
\sigma_\pm(k)=c_0k\pm\sqrt{g |k|}=c_0\left(k\pm \sqrt{\kappa|k|}\right),\quad\text{where}\quad k\in \mathbb{R}.
\end{equation}
We plot $\sigma_\pm(k)$ for $c_0=\kappa=1$ in Figure \ref{fig:dispersion}.
\begin{figure}[htbp]
\begin{center}
\includegraphics[scale=0.3]{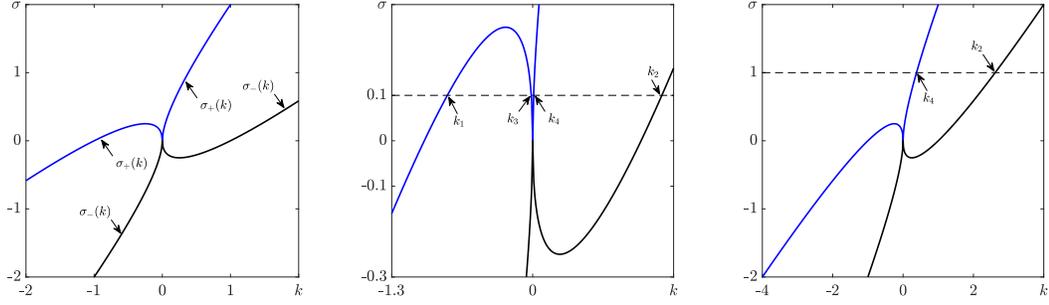}
\caption{Left: The graphs of $\sigma_+(k)$ (blue) and $\sigma_-(k)$ (black) for $c_0=\kappa=1$. Middle: When $0<\sigma<\sigma_c$, $\sigma_\pm(k)=\sigma$ have four roots $k_j(\sigma)$, $j=1,2,3,4$. Right: When $\sigma>\sigma_c$, $\sigma_\pm(k)=\sigma$ have two roots $k_j(\sigma)$, $j=2,4$.}\label{fig:dispersion}
\end{center}
\end{figure}

Let $k_c=-\frac{\kappa}{4}$ and $\sigma_c=\sigma_+(k_c)=\frac{c_0\kappa}{4}$ so that $\sigma'_+(k_c)=0$.
Let $k_2(\sigma)>0$ denote the simple root of $\sigma_-(k)=\sigma(\geq0)$, and let $k_4(\sigma)>0$ be the simple root of $\sigma_+(k)=\sigma(>0)$, and $k_4(0)=0$. When $0\leq\sigma\leq\sigma_c$, let $k_1(\sigma)\leq k_3(\sigma)\leq0$ be the other two roots of $\sigma_+(k)=\sigma$. See the middle and right panels of Figure~\ref{fig:dispersion}. Thus:
\begin{enumerate}
\item When $\sigma=0$, $k_j(0)=(-1)^j\kappa$, $j=1,2$, and $k_j(0)=0$, $j=3,4$; 
\item When $0<\sigma<\sigma_c$, $\sigma_-(k_2)=\sigma_+(k_j)=\sigma$, $j=1,3,4$, and $k_1<k_3<0<k_4<k_2$;
\item When $\sigma=\sigma_c$, $k_j(\sigma_c)=k_c$, $j=1,3$, and $\sigma_-(k_2)=\sigma_+(k_4)=\sigma_c$, $k_1=k_3<0<k_4<k_2$;
\item When $\sigma>\sigma_c$, $\sigma_-(k_2)=\sigma_+(k_4)=\sigma$ and $0<k_4<k_2$.
\end{enumerate}


\begin{lemma}[Spectrum of $\mathbf{L}({i\sigma})$]\label{lem:eps=0} 
When $\sigma=0$, $ik_j(0)=(-1)^ji \kappa$, $j=1,2$, are simple eigenvalues of $\mathbf{L}(0):{\rm dom}(\mathbf{L})\subset Y\to Y$, and
\be\label{def:phi12}
\ker(\mathbf{L}(0)-ik_j(0)\mathbf{1})={\rm span}\{\boldsymbol{\phi}_j(0)\},\quad\boldsymbol{\phi}_j(0)=\begin{pmatrix}
e^{\kappa y}\\(-1)^jie^{\kappa y}/c_0\\(-1)^ji/c_0\end{pmatrix},
\ee
where $\mathbf{1}$ denotes the identity operator.
Also, $ik_j(0)=0$, $j=3,4$, is an eigenvalue of $\mathbf{L}(0)$ with algebraic multiplicity $2$ and geometric multiplicity $1$, and 
\be \label{phi3}
\ker(\mathbf{L}(0))={\rm span}\{\boldsymbol{\phi}_3(0)\},\quad
\boldsymbol{\phi}_3(0)=\begin{pmatrix} g/c_0 \\ 0 \\ 0 \end{pmatrix}.
\ee
When $0<\sigma<\sigma_c$, $ik_j(\sigma)$, $j=1,2,3,4$, are simple eigenvalues of $\mathbf{L}(i\sigma)$, and
\ba\label{def:phi1-4}
&\ker(\mathbf{L}(i\sigma)-ik_j(\sigma)\mathbf{1})=\ker((\mathbf{L}(i\sigma)-ik_j(\sigma)\mathbf{1})^2)
={\rm span}\{\boldsymbol{\phi}_j(\sigma)\},\\
&\boldsymbol{\phi}_j(\sigma)=\begin{pmatrix}
e^{|k_j(\sigma)|y}\\i(k_j(\sigma)c_0-\sigma)e^{|k_j(\sigma)|y}/g\\i(k_j(\sigma)c_0-\sigma)/g\end{pmatrix}.
\ea
When $\sigma>\sigma_c$, $ik_j(\sigma)$, $j=2,4$, are simple eigenvalues of $\mathbf{L}(i\sigma)$, and \eqref{def:phi1-4} holds. \\
Let $Y(\sigma)$ denote the eigenspace of $\mathbf{L}(i\sigma)$ associated with its finitely many and purely imaginary eigenvalues and $\boldsymbol{\Pi}(\sigma): {\rm dom}(\mathbf{L})\subset Y \to Y(\sigma)$
be the projection of ${\rm dom}(\mathbf{L})$ onto $Y(\sigma)$, which commutes with $\mathbf{L}(i\sigma)$.
\end{lemma}
\begin{remark}
When $\sigma=0$, $\boldsymbol{\phi}_4(0)=\begin{pmatrix} 0 & 1 & 1 \end{pmatrix}^T$ is an generalized eigenvector corresponding to eigenvalue $0$. Later, we only reduce the spectral problem to the finite dimensional space spanned by $\{\boldsymbol{\phi}_1(0),\boldsymbol{\phi}_2(0),\boldsymbol{\phi}_3(0)\}$. This is because technically it is not possible to find the expansion of the reduction function when $j=4$. That is when computing for $\mathbf{w}^{1,0}_4$, one find that it is necessary to use a term involving $y^2$ in the its first entry, just like \cite{hur2021unstable} Appendix E first entry of $\mathbf{w}_3^{1,0}$. But now, $y^2\notin L^2(-\infty,0)$.
\end{remark}
\subsection{Computation of \texorpdfstring{$\boldsymbol{\Pi}(\sigma)$}{Lg}}\label{sec:proj}
We begin by constructing the adjoint of $\mathbf{L}(\lambda):{\rm dom}(\mathbf{L})\subset Y\to Y$. For $\mathbf{u}_1:=\begin{pmatrix}\phi_1+\phi_{1,\infty}\\ u_1+u_{1,\infty} \\ \eta_1\end{pmatrix}, 
\mathbf{u}_2:=\begin{pmatrix}\phi_2+\phi_{2,\infty}\\ u_2+u_{2,\infty}\\ \eta_2 \end{pmatrix} \in Y$ where $\phi_{1,2}\in H^1(-\infty,0)$, $u_{1,2}\in L^2(-\infty,0)$, and $\phi_{1,2,\infty},u_{1,2,\infty},\eta_1,\eta_2\in \mathbb{C}$, the inner product of $\mathbf{u}_1$ with $\mathbf{u}_2$ is given by
\be\label{def:inner}
\langle \mathbf{u}_1, \mathbf{u}_2\rangle
=\int^0_{-\infty}\phi_1\phi_2^*+{\phi_1}_y{\phi_2}_y^*~dy+\phi_{1,\infty}\phi_{2,\infty}^*+\int^0_{-\infty} u_1u_2^*~dy+u_{1,\infty}u_{2,\infty}^*+\eta_1\eta_2^*,
\ee
where the asterisk means complex conjugation. For $\mathbf{u}_1\in {\rm dom}(\mathbf{L})\subset Y$ and $\mathbf{u}_2\in Y$, 
\begin{align*}
&\langle \mathbf{L}(\lambda)\mathbf{u}_1, \mathbf{u}_2\rangle
=\left\langle\begin{pmatrix}(\lambda\phi_1+gu_1)/c_0+(\lambda\phi_{1,\infty}+gu_{1,\infty})/c_0\\-(c_0^2{\phi_1}_{yy}+\lambda^2\phi_1+g\lambda u_1)/(g c_0)+(\lambda^2\phi_{1,\infty}+g\lambda u_{1,\infty})/(g c_0)\\(\lambda\eta_1-{\phi_1}_y(0))/c_0\end{pmatrix},
\begin{pmatrix} \phi_2+\phi_{2,\infty} \\ u_2+u_{2,\infty} \\ \eta_2 \end{pmatrix}\right\rangle \\
=&\int^0_{-\infty}\phi_1(\lambda^*\phi_2/c_0)^*+{\phi_1}_y(\lambda^*{\phi_2}_y/c_0)^*+g u_1\phi_2^*/c_0-g u_1{\phi_2}_{yy}^*/c_0+c_0{\phi_1}_y{u_2}_y^*/g~dy\\
&-\int^0_{-\infty}\phi_1({\lambda^*}^2u_2/(gc_0))^*+u_1(\lambda^*u_2/c_0)^*~dy+g(\eta_1-u_{1,\infty}){\phi_2}_y^*(0)/c_0\\&-g u_1(-\infty){\phi_2}_y^*(-\infty)/c_0-{\phi_1}_y(0)(c_0u_2(0)/g+\eta_2/c_0)^*+\lambda\eta_1\eta_2^*/c_0\\&+(\lambda\phi_{1,\infty}+gu_{1,\infty})\phi_{2,\infty}^*/c_0+(\lambda^2\phi_{1,\infty}+g\lambda u_{1,\infty})u_{2,\infty}^*/(g c_0)\\
=&\int^0_{-\infty}\phi_1(\lambda^*\phi_2/c_0)^*+{\phi_1}_y(\lambda^*{\phi_2}_y/c_0)^*+g u_1\phi_2^*/c_0-g u_1{\phi_2}_{yy}^*/c_0+c_0{\phi_1}_y{u_2}_y^*/g~dy\\
&-\int^0_{-\infty}\phi_1({\lambda^*}^2u_2/(gc_0))^*+u_1(\lambda^*u_2/c_0)^*~dy+g\eta_1{\phi_2}_y^*(0)/c_0+\lambda\eta_1\eta_2^*/c_0\\
&+(\lambda\phi_{1,\infty}+gu_{1,\infty})\phi_{2,\infty}^*/c_0+(\lambda^2\phi_{1,\infty}+g\lambda u_{1,\infty})u_{2,\infty}^*/(g c_0)\\
=&\left\langle \begin{pmatrix} \phi_1+\phi_{1,\infty} \\ u_1+u_{1,\infty} \\ \eta_1 \end{pmatrix}, 
\begin{pmatrix}\lambda^*\phi_2/c_0+c_0u_2/g+\lambda^*\phi_{2,\infty}/c_0+{\lambda^*}^2u_{2,\infty}/(gc_0)\\ g\phi_2/c_0-g{\phi_2}_{yy}/c_0-\lambda^*u_2/c_0+g\phi_{2,\infty}/c_0+\lambda^*u_{2,\infty}/c_0-g{\phi_2}_y(0)/c_0\\g{\phi_2}_y(0)/c_0+\lambda^*\eta_2/c_0\end{pmatrix}\right\rangle\\
&-\int^0_{-\infty}\phi_1(c_0^2+{\lambda^*}^2)u_2^*/(gc_0)~dy \\
=&\left\langle \begin{pmatrix} \phi_1+\phi_{1,\infty} \\ u_1+u_{1,\infty} \\ \eta_1 \end{pmatrix}, 
\begin{pmatrix}\lambda^*\phi_2/c_0+c_0u_2/g+\lambda^*\phi_{2,\infty}/c_0+{\lambda^*}^2u_{2,\infty}/(gc_0)\\ g\phi_2/c_0-g{\phi_2}_{yy}/c_0-\lambda^*u_2/c_0+g\phi_{2,\infty}/c_0+\lambda^*u_{2,\infty}/c_0-g{\phi_2}_y(0)/c_0\\g{\phi_2}_y(0)/c_0+\lambda^*\eta_2/c_0\end{pmatrix}\right\rangle\\
&+\left\langle \begin{pmatrix} \phi_1 \\ u_1 \\ \eta_1 \end{pmatrix},
\begin{pmatrix}\phi_p\\ u_p\\ \eta_p \end{pmatrix}\right\rangle
=:\langle \mathbf{u}_1,\mathbf{L}(\lambda)^\dag\mathbf{u}_2\rangle,
\end{align*}
where $\mathbf{L}(\lambda)^\dag$ denotes the adjoint of $\mathbf{L}(\lambda)$. Here the first equality uses \eqref{def:L}, and the second equality uses \eqref{def:inner} and follows after integration by parts, because if $\mathbf{u}_1\in {\rm dom}(\mathbf{L})$ then $\eta_1=u_1(0)+u_{1,\infty}$ and ${\phi_1}_y(-\infty)=0$ (see \eqref{def:L}). The third equality follows, provided that 
\begin{equation}\label{eqn:bdry;adj}
u_2(0)+\kappa\eta_2=0\quad\text{and}\quad {\phi_2}_y(-\infty)=0,
\end{equation}
so that the inner product is continuous with respect to $\phi_1\in H^1(-\infty,0)$ and $u_1\in L^2(-\infty,0)$ (see \eqref{def:Y}),
and the fourth equality uses \eqref{def:inner}. The fifth equality follows, provided that 

\begin{align*}
-\int^0_{-\infty}\phi_1(c_0^2+{\lambda^*}^2)u_2^*/(gc_0)~dy &=
\left\langle \begin{pmatrix} \phi_1 \\ u_1 \\ \eta_1 \end{pmatrix}, 
\begin{pmatrix}\phi_p\\ u_p\\ \eta_p \end{pmatrix}\right\rangle \\
&=\int^0_{-\infty} (\phi_1\phi_p^*+{\phi_1}_y{\phi_p}_y^*)~dy+\int^0_{-\infty} u_1u_p^*~dy+\eta_1\eta_p^*\\
&=\int^0_{-\infty}\phi_1(\phi_p-{\phi_p}_{yy})^*~dy+\phi_1(0){\phi_p}_y^*(0)-\phi_1(-\infty){\phi_p}_y^*(-\infty),
\end{align*}
where the last equality assumes that $u_p=0$ and $\zeta_p=0$, because the left side does not depend on $u_1$ or $\eta_1$, and it follows after integration by parts. This works, provided that 
\begin{equation}\label{eqn:up}
\left\{\begin{aligned}
&{\phi_p}_{yy}-\phi_p=(c_0^2+{\lambda^*}^2)u_2/(gc_0)\quad\text{for $y<0$},\\
&{\phi_p}_y(0)=0,\\
&{\phi_p}_y(-\infty)=0.
\end{aligned}\right.
\end{equation}
To recapitulate, 
\[
\mathbf{L}(\lambda)^\dag:{\rm dom}(\mathbf{L}^\dag) \subset Y\to Y,
\] 
where 
\[
L(\lambda)^\dag\begin{pmatrix}\phi_2+\phi_{2,\infty}\\ u_2+u_{2,\infty}\\ \eta_2 \end{pmatrix}
=\begin{pmatrix}\lambda^*\phi_2/c_0+c_0u_2/g+\lambda^*\phi_{2,\infty}/c_0+{\lambda^*}^2u_{2,\infty}/(gc_0)+\phi_p\\ g\phi_2/c_0-g{\phi_2}_{yy}/c_0-\lambda^*u_2/c_0+g\phi_{2,\infty}/c_0+\lambda^*u_{2,\infty}/c_0-g{\phi_2}_y(0)/c_0\\g{\phi_2}_y(0)/c_0+\lambda^*\eta_2/c_0\end{pmatrix}\\
\]
\begin{equation}\label{def:up}
\phi_p(y)=-\frac{c_0^2+{\lambda^*}^2}{gc_0}\int_{y}^0\cosh(y')u_2(y')dy'e^{y}-\frac{c_0^2+{\lambda^*}^2}{gc_0}\int_{-\infty}^ye^{y'}u_2(y')dy'\cosh(y),
\ee 
and
\[
{\rm dom}(\mathbf{L}^\dag)=\{(\phi_2+\phi_{2,\infty}, u_2+u_{2,\infty},\eta_2)^T \in H^2_c(-\infty,0)\times H^1_c(-\infty,0)\times \mathbb{C}:
u_2(0)+\kappa\eta_2=0, {\phi_2}_y(-\infty)=0\}.
\]
When $\sigma=0$, we find\be\label{def:Pi0} 
\boldsymbol{\Pi}(0)\mathbf{u}=\sum_{i=1}^3\langle \mathbf{u},\boldsymbol{\psi}_i(0)\rangle\boldsymbol{\phi}_i(0),
\ee 
where
\be 
\boldsymbol{\psi}_1(0)=\begin{pmatrix}
\frac{e^y - \kappa e^{\kappa y}}{1-\kappa^2}\\ic_0(\kappa e^{\kappa y}+1)\\-ic_0
\end{pmatrix},\quad \boldsymbol{\psi}_2(0)=\begin{pmatrix}
\frac{e^y - \kappa e^{\kappa y}}{1-\kappa^2}\\-ic_0(\kappa e^{\kappa y}+1)\\ic_0
\end{pmatrix},\quad  \boldsymbol{\psi}_3(0)=\begin{pmatrix}
\frac{c_0}{g}\\0\\0
\end{pmatrix}.
\ee 
\begin{remark*}
The first entry of $\boldsymbol{\psi}_1(0)$ ($\boldsymbol{\psi}_2(0)$) appears to be not defined when $\kappa=1$. But $\kappa=1$ turns out to be a removable singularity.
\end{remark*}
When $\sigma>\sigma_c$, we infer from Lemma~\ref{lem:eps=0} that $-ik_j(\sigma)$, $j=2,4$, are simple eigenvalues of $\mathbf{L}(i\sigma)^\dag$, and a straightforward calculation reveals that the corresponding eigenfunctions are
\be \label{def:psi24}
\boldsymbol{\psi}_j(\sigma)=\frac{1}{{c_{0}}^2{k_j}^2-2c_{0}k_j\sigma +\kappa c_{0}\sigma +\sigma ^2}\begin{pmatrix} -\frac{k_j\kappa \left({c_{0}}^2-\sigma ^2\right)}{\left({k_j}^2-1\right)}e^{y}-\frac{\kappa \left(\sigma ^2-{c_{0}}^2{k_j}^2\right)}{\left({k_j}^2-1\right)}e^{k_j y}+\frac{c_{0}\kappa \sigma ^2}{\left(2\sigma -c_{0}k_j\right)}\\
 ic_0^2\kappa^2(\sigma - c_0k_j)e^{k_jy}-\frac{ic_0^3\kappa^2(\sigma - c_0k_j)}{2\sigma - c_0k_j}\\
-ic_0^2\kappa(\sigma - c_0k_j)\end{pmatrix},
\ee 
so that  $\langle \boldsymbol{\phi}_{j}(\sigma),\boldsymbol{\psi}_{j'}(\sigma)\rangle=\delta_{jj'}$, $j,j'=2,4$, where $\boldsymbol{\phi}_j(\sigma)$, $j=2,4$, are in \eqref{def:phi1-4}. Thus
\begin{equation}\label{def:Pi;high}
\boldsymbol{\Pi}(\sigma)\mathbf{u}=\langle \mathbf{u},\boldsymbol{\psi}_2(\sigma)\rangle\boldsymbol{\phi}_2(\sigma)
+\langle \mathbf{u},\boldsymbol{\psi}_4(\sigma)\rangle\boldsymbol{\phi}_4(\sigma).
\end{equation}
When $0<\sigma\leq\sigma_c$, we proceed likewise to define $\boldsymbol{\Pi}(\sigma)$. We do not include the formulae here.

Based on the center manifold reduction by Mielke \cite{Mielke;reduction,alma99954915956405899}, we reduce the spectral problem to
\be \label{eqn:LB;u1}
{\mathbf{v}}_x=\mathbf{L}(i\sigma)\mathbf{v}
+\boldsymbol{\Pi}(\sigma)\mathbf{B}(x;\sigma,\delta,\eps)
(\mathbf{v}(x)+\mathbf{w}(x,\mathbf{v}(x);\sigma,\delta,\eps)).
\ee 
where $\mathbf{v}\in Y(\sigma)$ and the reduction function $\mathbf{w}(x,\mathbf{v}(x);\sigma,\delta,\eps))$ satisfies
$$
{\mathbf{w}}_x=\mathbf{L}(i\sigma)\mathbf{w}
+(\mathbf{1}-\boldsymbol{\Pi}(\sigma))\mathbf{B}(x;\sigma,\delta,\eps)(\mathbf{v}(x)+\mathbf{w}(x)).
$$
Let $\mathbf{a}(x)$ be the coordinate of $\mathbf{v}(x)$ with respect to the ordered basis of $Y(\sigma)$ given in Lemma \eqref{lem:eps=0}. We further rewrite \eqref{eqn:LB;u1} as
\be\label{eqn:A}
\mathbf{a}_x=\mathbf{A}(x;\sigma,\delta,\eps)\mathbf{a},
\ee 

\begin{definition}[The periodic Evans function]\label{def:Evans}\rm
For $\lambda=i\sigma+\delta$, $\sigma\in \mathbb{R}$, $\delta\in\mathbb{C}$ and $|\delta|\ll1$, for $\eps\in\mathbb{R}$ and $|\eps|\ll1$, let $\mathbf{X}(x;\sigma,\delta,\eps)$ denote the fundamental solution of \eqref{eqn:A} such that $\mathbf{X}(0;\sigma,\delta,\eps)=\mathbf{I}$, where $\mathbf{I}$ is the identity matrix. Let $\mathbf{X}(T;\sigma,\delta,\eps)$ be the {\em monodromy matrix} for \eqref{eqn:A}, and for $k\in\mathbb{R}$,
\be\label{def:Delta}
\Delta(\lambda,k;\eps)=\det(e^{ikT}\mathbf{I}-\mathbf{X}(T;\sigma,\delta,\eps))
\ee 
the {\em periodic Evans function}, where $T=2\pi/\kappa$ is the period of a Stokes wave.
\end{definition}
\begin{corollary}[spectrum of $\mathcal{L}(0)$, dispersion relation]\label{cor_dispersion}
For constant wave $\eps=0$, the periodic Evans function \eqref{def:Delta} satisfies
\be 
\Delta(i\sigma,k_j(\sigma);0)=0,\quad 
\ee 
\end{corollary}
\begin{proof}
When setting $\delta=\eps=0$, because $\mathbf{B}(x;\sigma,0,0)=0$, \eqref{eqn:LB;u1} reduces to ${\mathbf{v}}_x=\mathbf{L}(i\sigma)\mathbf{v}$ whose solutions are discussed in Section \ref{sec:eps=0}.
\end{proof}
We thereby study the nearby root $(i\sigma+\delta,k_j(\sigma)+dk,\eps)$ of the periodic Evans function $\Delta$ \eqref{def:Delta} for $\delta\in \mathbb{C}$, $dk,\eps\in \mathbb{R}$ and $|\delta|,|dk|,|\eps|\ll 1$.
And we expand the fundamental solution $\mathbf{X}(x;\sigma,\delta,\eps)$ of \eqref{eqn:A} as
\be 
\label{def:X;exp}
\mathbf{X}(x;\sigma,\delta,\eps)=\sum_{m+n=0}^{\infty}\mathbf{a}^{(m,n)}(x;\sigma)\delta^m\eps^n\in \mathbb{C}^{dim(Y(\sigma))\times dim(Y(\sigma))}.
\ee 
where $\mathbf{a}^{(m,n)}=(a_{jk}^{(m,n)}(x))$, $j,k=1,2,\ldots,\dim(Y(\sigma))$ and $m,n=0,1,2,\dots$.
\section{The Benjamin--Feir instability}\label{sec:BF}
In this section we set $\sigma=0$ and study the roots of periodic Evans function near the origin.
\begin{lemma} By direct computation, we find 
\ba\label{eqn:amn}
&\mathbf{a}^{(0,0)}(T)=Id,&& \mathbf{a}^{(0,1)}(T)=\mathbf{0},\\
&\mathbf{a}^{(1,0)}(T)=diag\left\{\frac{4\pi}{c_0\kappa},\frac{4\pi}{c_0\kappa},\frac{2\pi}{c_0\kappa}\right\},&&\mathbf{a}^{(2,0)}(T)=diag\left\{\frac{2\pi(4\pi - i)}{c_0^2\kappa^2},\frac{2\pi(4\pi + i)}{c_0^2\kappa^2},\frac{2\pi^2}{c_0^2\kappa^2}\right\},\\
&\mathbf{a}^{(1,1)}(T)=\begin{pmatrix} 0 & 0 & \frac{2\pi\kappa}{c_0}\\ 0 & 0 & \frac{2\pi\kappa}{c_0}\\0&0&0\end{pmatrix},&&\mathbf{a}^{(0,2)}(T)=\begin{pmatrix} -\frac{2i\pi\kappa^2}{c_0^2} & \frac{2i\pi\kappa^2}{c_0^2} & 0\\ -\frac{2i\pi\kappa^2}{c_0^2} & \frac{2i\pi\kappa^2}{c_0^2} & 0\\0&0&0\end{pmatrix}
\ea
\end{lemma}
\begin{proof}
The proof is by directly computation following the steps outlined in \cite{hur2021unstable_cap} section 3.6.
\end{proof}
Let
\[
\text{$\delta\in\mathbb{C}$ and $|\lambda|\ll1$},\quad 
\text{$k=p\kappa+\gamma$, $p\in\mathbb{Z}$, $\gamma\in\mathbb{R}$ and $|\gamma|\ll1$},\quad
\text{$\eps\in\mathbb{R}$ and $|\eps|\ll1$}.
\]
Substituting \eqref{eqn:amn} \eqref{def:X;exp} into \eqref{def:Delta} yields
\ba \label{eqn:evans0}
\Delta(\delta,p\kappa+\gamma;\eps)=&\sum_{l=0}^3d^{(l,3-l,0)}\delta^{l}\gamma^{3-l}+\sum_{l=0}^3d^{(l,3-l,2)}\delta^{l}\gamma^{3-l}\eps^2+o((|\delta|+|\gamma|)^3+|\eps|^2),
\ea
as $\lambda,\gamma,\eps\to0$, where $d^{(\ell,m,n)}$, $\ell,m,n=0,1,2,\dots$ can be determined in terms of $a_{jk}^{(m,n)}(T)$, $j,k=1,2,3$ and $m,n=0,1,2,\dots$, where $T=2\pi/\kappa$ is the period of a Stokes wave. 
When $\eps=0$, by Corollary \ref{cor_dispersion}, $\Delta(i\sigma(k),k;0)=0$ for any $ k\in\mathbb{R}$, where $\sigma$ is in \eqref{eqn:sigma}. Particularly, $\Delta(\lambda_j(k_j(0),0),k_j(0);0)=0$, $j=1,2,3$, where $k_j(0)=(-1)^j\kappa$ for $j=1,2$ and $k_3(0)=0$.
In other words, $\lambda=0$ and $k=k_j(0)$, $j=1,2,3$, are the three roots of $\Delta(\cdot,\cdot\,;0)=0$. For $\gamma, \eps\in\mathbb{R}$ and $|\gamma|, |\eps|\ll1$, we are interested in determining $\lambda_j(k_j(0)+\gamma,\eps)$ such that 
\begin{equation}\label{eqn:lambda}
\lambda_j(k_j(0)+0,0)=0\quad\text{and}\quad 
\Delta(\lambda_j(k_j(0)+\gamma,\eps), k_j(0)+\gamma;\eps)=0.
\end{equation}
Let
\begin{equation}\label{def:lambda(gamma,eps)}
\lambda_j(k_j(0)+\gamma,\eps)=
\alpha^{(1,0)}_j\gamma+\alpha^{(1,1)}_j\gamma\eps
+o(|\gamma|+ |\gamma||\eps|),\quad j=1,2,\quad \text{as $\gamma,\eps\to0$.}
\end{equation}
We pause to remark \eqref{def:lambda(gamma,eps)} is not valid for $j=3$ for \eqref{eqn:sigma} is not analytic at $k=0$. However, the dispersion relation of Stoke waves in water of finite depth is analytic at $k=0$.
Substituting \eqref{def:lambda(gamma,eps)} into \eqref{eqn:evans0}, 
after straightforward calculations, we learn that $\gamma^3$ is the leading order whose coefficient reads
\begin{equation}\label{eqn:alpha10}
\begin{aligned}
\frac{8i\pi^3(c_0 + i\alpha^{(1,0)}_j)(2\alpha^{(1,0)}_j - ic_0)^2}{c_0^3\kappa^3}
\end{aligned}
\end{equation}
Solving the latter equation of \eqref{eqn:lambda} at the order of $\gamma^3$, \eqref{eqn:alpha10} must vanish, whence
\be\label{def:alpha10'}
\alpha^{(1,0)}_j=\frac{ic_0}{2}\quad \text{or}\quad\alpha^{(1,0)}_j=ic_0
\ee
On the other hand, \eqref{def:alpha10'} must agree with power series expansions of \eqref{def:sigma} about $\pm\kappa$. Thus 
\begin{equation}\label{def:alpha10}
\alpha^{(1,0)}_j=\frac{ic_0}{2} \quad\text{for}\quad j=1,2.
\end{equation}

Substituting \eqref{def:lambda(gamma,eps)} into \eqref{eqn:evans0} and evaluating at \eqref{def:alpha10}, after straightforward calculations, we verify that the $\gamma^2\eps^2$ term vanishes, and the coefficient of $\gamma^3\eps^2$ reads
\ba \label{eqn:f}
\frac{2i\pi^3(8(\alpha^{(1,1)}_j)^2 - \kappa^2)}{c_0^2\kappa^3}
\ea

\begin{proof}[proof of Theorem \ref{thm:BF}]
The coefficient of $\gamma^3\eps^2$ must vanish, whence
\begin{equation}\label{def:alpha11}
\alpha^{(1,1)}_j=\pm \frac{\kappa}{2\sqrt{2}}, \quad j=1,2.
\end{equation}
Hence, \eqref{def:lambda(gamma,eps)} bifurcate off the imaginary axis at $\eps k$ order for $j=1,2$. This completes the proof. 
\end{proof}
\section{Instability at non-zero resonant frequency}
Our previous analysis \cite{hur2021unstable,hur2021unstable_cap} for waves in water of finite depth shows instability can also occur at some non-zero resonant frequencies. 
\begin{definition}[Resonant frequency \cite{hur2021unstable_cap}]
Let $ik_j(\sigma)$ be eigenvalues of $\mathbf{L}(i\sigma)$ defined in Section \ref{sec:eps=0}. We call $(k_i(\sigma),k_j(\sigma),N)$ a pair of $N$-resonant eigenvalues of $\mathbf{L}(i\sigma)$ provided that $k_i(\sigma)-k_j(\sigma)=N\kappa$. And, we call the following set $\mathcal{R}(\sigma)$ the set of pairs of $N$-resonant eigenvalues of $\mathbf{L}(i\sigma)$.
\be
\mathcal{R}(\sigma):=\{(k_i(\sigma),k_j(\sigma),N):k_i(\sigma)-k_j(\sigma)=N\kappa, \;\text{for some order $N\in \mathbb{Z}^+$}\}.
\ee
If $\mathcal{R}(\sigma)\neq \emptyset$, we call $\lambda=i\sigma$ a resonant frequency of the wave, otherwise, we call $\lambda=i\sigma$ a non-resonant frequency.
\end{definition}
By the simplicity of \eqref{def:sigma}, we can study a pair of $N$-resonant eigenvalues explicitly. By symmetry, assume $k_i(\sigma)-k_j(\sigma)=N\kappa$ for some $\sigma>0$ and both $(k_i(\sigma),\sigma)$ and $(k_j(\sigma),\sigma)$ are on  $\sigma_+$-curve. See Figure \ref{fig:dispersion}. Then, we have 
$$k_i+\sqrt{\kappa|k_i|}=k_j+\sqrt{\kappa|k_j|},\quad k_i-k_j=N\kappa.$$
Clearly, $i=4$, $j=3$ or $j=1$, then we have $N\sqrt{\kappa}=\sqrt{-k_j}-\sqrt{k_i}<\sqrt{-k_j}$ implying $N<1$. Contradiction. Hence, $(k_i(\sigma),\sigma)$ and $(k_j(\sigma),\sigma)$ cannot both be on  $\sigma_+$-curve. Indeed, some further analysis shows the only possible resonance is between $k_2$ and $k_4$. Solving
$$\sigma=c_0(k_2-\sqrt{\kappa k_2})=c_0(k_4+\sqrt{\kappa k_4}),\quad k_2-k_4=N\kappa$$
yields
\be 
\label{resonance_k2k4}
\sigma=\frac{N^2-1}{4}\kappa c_0,\quad k_2=\frac{(N+1)^2}{4}\kappa,\quad k_4=\frac{(N-1)^2}{4}\kappa.
\ee 
Because $\frac{N^2-1}{4}\kappa c_0>\sigma_c$, $Y(\sigma)$ is two dimensional.
\begin{lemma}[\cite{hur2021unstable,hur2021unstable_cap}]\label{coefficients_highn1}
At the resonant frequency $\lambda=i\frac{N^2-1}{4}\kappa c_0$ where $k_2(\sigma)-k_4(\sigma)=N\kappa$, computations show
\ba 
\label{coeff_1}
\mathbf{a}^{(0,0)}(T)=&e^{ik_4T}\begin{pmatrix}1&0\\0&1\end{pmatrix},\quad 
\mathbf{a}^{(1,0)}(T)=\begin{pmatrix}\frac{2(-1)^{\tfrac{1}{2}(N + 1)^2}\pi(N + 1)}{Nc_0\kappa}&0\\0&\frac{2(-1)^{\tfrac{1}{2}(N - 1)^2}\pi(N - 1)}{Nc_0\kappa}\end{pmatrix},\\
\mathbf{a}^{(0,1)}(T)=&\begin{pmatrix}0&0\\0&0\end{pmatrix},
\ea
and 
\be \label{coeff_2}
\text {for $N=2$,}\quad 
\mathbf{a}^{(0,2)}(T)=\begin{pmatrix}-\frac{27\pi\kappa^2}{8c_0^2}&0\\0&\frac{\pi\kappa^2}{16c_0^2}\end{pmatrix},\quad \text {for $N\ge 3$,}\quad \mathbf{a}^{(0,2)}(T)=diag\{\mathbf{a}^{(0,2)}_{11},\mathbf{a}^{(0,2)}_{22}\},
\ee 
where $\mathbf{a}^{(0,2)}_{11},\mathbf{a}^{(0,2)}_{22}\in i(-1)^{\tfrac{1}{2}(N + 1)^2}\mathbb{R}$.
\end{lemma}
\begin{proof}
The proof is by directly computation following the steps outlined in \cite{hur2021unstable_cap} section 3.6.
\end{proof}
\begin{remark}
For $N=2$, different from the case of finite depth, the off-diagonal entries of $\mathbf{a}^{(0,2)}(T)$ vanish, suggesting there should be no instability at the resonant frequency of order $2$.
\end{remark}
Let
\[
\text{$\lambda=i\sigma+\delta$, $\delta\in\mathbb{C}$ and $|\delta|\ll1$},\quad
\text{$k=k_j(\sigma)+p\kappa+\gamma$, $j=2,4$, $p\in\mathbb{Z}$, $\gamma\in\mathbb{R}$ and $|\gamma|\ll1$},
\]
$\eps\in\mathbb{R}$ and $|\eps|\ll1$, and we recall the result of Lemma~\ref{coefficients_highn1} to arrive at 
\ba  \label{expanddelta2}
&\Delta(i\sigma+\delta,k_j(\sigma)+p\kappa+\gamma;\eps)\\
=&a_{11}^{(1,0)}a_{22}^{(1,0)}\delta^2-T^2e^{2ik_4(\sigma)T}\gamma^2+\det(\mathbf{a}^{(0,2)}(T))\eps^4-iTe^{ik_4(\sigma)T}(a_{11}^{(1,0)}+a_{22}^{(1,0)})\delta\gamma\\
&+(a_{11}^{(0,2)}a_{22}^{(1,0)}+a_{22}^{(0,2)}a_{11}^{(1,0)})\delta\eps^2
-iTe^{ik_4(\sigma)T}(a_{11}^{(0,2)}+a_{22}^{(0,2)})\gamma\eps^2\\
&-iTe^{ik_4(\sigma)T}(a_{11}^{(1,1)}+a_{22}^{(1,1)})\delta \gamma \eps+o(|\delta|^2+|\gamma|^2+|\eps|^4+|\delta \gamma|+|\delta| |\eps|^2+|\gamma||\eps|^2+|\delta\gamma||\eps|)
\ea 
as $\delta,\gamma,\eps\to0$. We seek $\lambda_{2j}(k_{2j}(\sigma)+\gamma,\eps)$, $j=1,2$, such that 
\[
\lambda_{2j}(k_{2j}(\sigma)+0,0)=i\sigma\quad\text{and}\quad 
\Delta(\lambda_{2j}(k_{2j}(\sigma)+\gamma,\eps),k_{2j}(\sigma)+\gamma;\eps)=0
\]
Let
\be \label{def:lambda(gamma,eps)H}
\lambda_{2j}(k_{2j}(\sigma)+\gamma,\eps)=i\sigma+
\alpha^{(1,0)}_{2j}\gamma+\alpha^{(0,2)}_{2j}\eps^2+o(|\gamma|+|\eps|^2),\quad j=1,2, \quad\text{as $\gamma, \eps\to0$,}
\ee     
where $\alpha^{(1,0)}_{2j}$ and $\alpha^{(0,2)}_{2j}$, $j=1,2$, are to be determined in terms of $a_{jk}^{(m,n)}$. Indeed, similar to \cite{hur2021unstable} (), $\alpha^{(0,2)}_{2j}$ shall solve
\begin{equation}\label{eqn:alpha02H}
(a_{11}^{(1,0)}a_{22}^{(1,0)})(\alpha^{(0,2)}_{2j})^2
+(a_{11}^{(0,2)}a_{22}^{(1,0)}+a_{11}^{(1,0)}a_{22}^{(0,2)})\alpha^{(0,2)}_{2j}
+a_{11}^{(0,2)}a_{22}^{(0,2)}=0.
\end{equation}
\begin{proof}[proof of Theorem \ref{thm:unstable2}]
Solving for $\alpha^{(0,2)}_{2j}$ in \eqref{eqn:alpha02H} yields the definition of $\rm{ind}_2$. By the formulas \eqref{coeff_1} and \eqref{coeff_2}, we find $$a_{11}^{(1,0)},a_{22}^{(1,0)}\in (-1)^{\tfrac{1}{2}(N+1)^2}\mathbb{R}\quad\text{and}\quad a_{11}^{(0,1)},a_{22}^{(0,1)}\in i(-1)^{\tfrac{1}{2}(N+1)^2}\mathbb{R}.$$
Therefore, $(a^{(0,1)}_{11}a^{(1,0)}_{22} - a^{(1,0)}_{11}a^{(0,1)}_{22})(a_{11}^{(1,0)}a_{22}^{(1,0)})^{-1}$ is purely imaginary yielding ${\rm ind}_{2}<0$.
\end{proof}

\bibliographystyle{amsplain}
\bibliography{wwbib}

\end{document}